\documentclass[12pt]{amsart}
\usepackage{amsfonts}
\usepackage{amssymb}
\usepackage{amsmath}

\newtheorem{theorem}{Theorem}[section]
\newtheorem{lemma}{Lemma}[section]
\newtheorem{proposition}{Proposition}[section]

\newtheorem{definition}{Definition}[section]
\newtheorem{example}{Example}[section]
\numberwithin{equation}{section}
\def\({\left ( }
\def\){\right )}
\def\<{\left < }
\def\>{\right >}

\begin{document}
\title[Some Types of Lightlike Submanifolds]{Screen Transversal Lightlike
Submanifolds of Golden Semi-Riemannian Manifolds}
\author{Feyza Esra Erdo\u{g}an(ADIYAMAN)}
\address{Faculty of Education, Department of Mathematics, Ad\i yaman
University, 02040 Ad\i yaman, TURKEY}
\email{ferdogan@adiyaman.edu.tr}

\begin{abstract}
The main purpose of the present paper is to
study the geometry of screen transversal lightlike submanifolds and radical
screen transversal lightlike submanifolds and screen transversal
anti-invariant lightlike submanifolds of Golden Semi-Riemannian manifolds.
We investigate the geometry of distributions and obtain necessary and
sufficient conditions for the induced connection on these manifolds to be
metric connection. We also obtain characterization of screen transversal
anti-invariant lightlike submanifolds of Golden Semi- Riemannian manifolds.
Finally, we give two examples.
\end{abstract}

\maketitle

\bigskip

\section{Introduction}

The Golden proportion, also called the Golden ratio, Divine ratio, Golden
section or Golden mean, has been well known since the time of Euclid. Many
objects alive in the natural world that possess pentagonal symmetry, such as
inflorescence of many flowers and prophylaxis objects have a numerical
description given by the Fibonacci numbers which are themselves based on the
Golden proportion. The Golden proportion has also been found in the
structure of musical compositions, in the ratios of harmonious sound
frequencies and in dimensions of the human body. From ancient times it has
played an important role in architecture and visual arts. \indent {}\newline
\noindent \line(1,0){100}\newline

{\footnotesize \textit{Key words and phrases}: Lightlike manifold, golden
semi Riemannian manifold, radical transversal lightlike submanifold, radical
screen transversal submanifold }\newline
{\footnotesize \textit{2010 Mathematics Subject Classification}. 53C15,
53C40, 53C50.}\newline
\newpage The Golden proportion and the Golden rectangle (which is spanned by
two sides in the Golden proportion) have been found in the harmonious
proportion of temples, churches, statues, paintings, pictures and fractals.
Golden structure was revealed by the golden proportion, which was
characterized by Johannes Kepler(1571-1630). The number $\phi $, which is
the real positive root of the equation%
\begin{equation*}
x^{2}-x-1=0,
\end{equation*}%
(hence, $\phi =\frac{1+\sqrt{5}}{2}\approx 1,618...$) is the golden
proportion.

The golden Ratio is fascinating topic that continually generated news ideas.
Golden Riemannian manifolds were introduced by Grasmereanu and Hretcanu\cite%
{GH}using Golden ratio. The authors also studied invariant submanifolds of
golden Riemannian manifold and obtained interesting result in \cite{GH2},%
\cite{GH3}. The integrability of such golden structures was also
investigated by Gezer, Cengiz and Salimov in \cite{GCS}. Moreover, the
harmonious of maps between golden Riemannian manifolds was studied in \cite%
{BM}. Furthermore, Erdo\u{g}an and Y\i ld\i r\i m in \cite{EC} studied
semi-invariant and totally umbilical semi invariant submanifolds of Golden
Riemannian manifolds, respectively. A semi-Riemannian manifold endowed with
a golden structure will be called a Golden semi-Riemannian manifold.
Precisely we can say that on (1,1)-tensor field $\breve{P}$ on a
m-dimensional semi-Riemannian manifold $(\breve{N},\breve{g})$ is Golden
structure if it satisfies the equation $\breve{P}^{2}=\breve{P}+I$, where $I$
is identity map on $\breve{N}$. Furthermore, 
\begin{equation*}
\breve{g}(\breve{P}W,U)=\breve{g}(W,\breve{P}U)
\end{equation*}%
the semi-Riemannian metric is called $\breve{P}$-compatible and $(\breve{N},%
\breve{g},\breve{P})$ is named a Golden semi-Riemannian manifold\cite{M}.

The theory of lightlike submanifolds of semi-Riemannian manifolds has a
important place in differential geometry. Lightlike submanifolds of a
semi-Riemannian manifolds has been studied by Duggal-Bejancu\cite{DB} and
Duggal and \c{S}ahin \cite{DS}. Indeed, lightlike submanifolds appear in
general relativity as some smooth parts of horizons of the Kruskal and Keer
black holes\cite{N}. Many authors studied the lightlike submanifolds in
different space, for example \cite{FC},\cite{BC},\cite{SBE} and \cite{FBR}.
Additionally, in a resent study, studied lightlike hypersurfaces of Golden
semi-Riemannian manifolds\cite{NE}.

Considering above information, in this paper, we introduce lightlike
submanifolds of Golden semi-Riemannian manifolds and studied their geometry.
The paper is organized as follows: In section2, we give basic information
needed for this paper. In section3 and section4 , we introduce Golden
semi-Riemannian manifold a long with its subclasses ( radical screen
transversal and screen transversal anti-invariant lightlike submanifolds)
and obtain some characterizations. We investigate the geometry of
distiributions and find necessary and sufficient conditions for induced
connection to be metric connection. Furthermore, we give two examples.

\section{Preliminaries}

A submanifold $\acute{N}^{m}$ immersed in a semi-Riemannian manifold $(%
\breve{N}^{m+k},\breve{g})$ is called a lightlike submanifold if it admits a
degenerate metric $g$\ induced from $\breve{g}$ whose radical distribution
which is a semi-Riemannian complementary distribution of $RadT\acute{N}$ is
of rank $r$, where $1\leq r\leq m.$ $RadT\acute{N}=T\acute{N}\cap T\acute{N}%
^{\perp }$ , where 
\begin{equation}
T\acute{N}^{\perp }=\cup _{x\in \acute{N}}\left\{ u\in T_{x}\breve{N}\mid 
\breve{g}(u,v)=0,\forall v\in T_{x}\acute{N}\right\} .
\end{equation}%
Let $S(T\acute{N})$ be a screen distribution which is a semi-Riemannian
complementary distribution of $RadT\acute{N}$ in $T\acute{N}$. i.e., $T%
\acute{N}=RadT\acute{N}\perp S(T\acute{N}).$

We consider a screen transversal vector bundle $S(T\acute{N}^{\bot }),$
which is a semi-Riemannian complementary vector bundle of $RadT\acute{N}$ in 
$T\acute{N}^{\bot }.$ Since, for any local basis $\left\{ \xi _{i}\right\} $
of $RadT\acute{N}$, there exists a lightlike transversal vector bundle $ltr(T%
\acute{N})$ locally spanned by $\left\{ N_{i}\right\} $ \cite{DB}. Let $tr(T%
\acute{N})$ be complementary ( but not orthogonal) vector bundle to $T\acute{%
N}$ in $T\breve{N}^{\perp }\mid _{\acute{N}}$. Then%
\begin{eqnarray*}
tr(T\acute{N}) &=&ltrT\acute{N}\bot S(T\acute{N}^{\bot }), \\
T\breve{N} &\mid &_{\acute{N}}=S(T\acute{N})\bot \lbrack RadT\acute{N}\oplus
ltrT\acute{N}]\perp S(T\acute{N}^{\bot }).
\end{eqnarray*}%
Although $S(T\acute{N})$ is not unique, it is canonically isomorphic to the
factor vector bundle $T\acute{N}/RadT\acute{N}$ \cite{DB}. The following
result is important for this paper.

\begin{proposition}
\cite{DB}. The lightlike second fundamental forms of a lightlike submanifold 
$N$ do not depend on $S(T\acute{N}),$ $S(T\acute{N}^{\perp })$ and $ltr(T%
\acute{N}).$
\end{proposition}

We say that a submanifold $(\acute{N},g,S(T\acute{N}),$ $S(T\acute{N}^{\perp
}))$ of $\breve{N}$ is

Case1: \ r-lightlike if $r<\min \{m,k\};$

Case2: \ Co-isotropic if $r=k<m;S(T\acute{N}^{\perp })=\{0\};$

Case3: \ Isotropic if $r=m=k;$ $S(T\acute{N})=\{0\};$

Case4: \ Totally lightlike if $r=k=m;$ $S(T\acute{N})=\{0\}=S(T\acute{N}%
^{\perp }).$\bigskip

The Gauss and Weingarten equations are:%
\begin{eqnarray}
\breve{\nabla}_{W}U &=&\nabla _{W}U+h(W,U),\forall W,U\in \Gamma (T\acute{N}%
),  \label{7} \\
\breve{\nabla}_{W}V &=&-A_{V}W+\nabla _{W}^{t}V,\forall W\in \Gamma (T\acute{%
N}),V\in \Gamma (tr(T\acute{N})),  \label{8}
\end{eqnarray}%
where $\left\{ \nabla _{W}U,A_{V}W\right\} $ and $\left\{ h\left( W,U\right)
,\nabla _{W}^{t}V\right\} $ belong to $\Gamma (T\acute{N})$ and $\Gamma (tr(T%
\acute{N})),$ respectively. $\nabla $ and $\nabla ^{t}$ are linear
connections on $\acute{N}$ and the vector bundle $tr(T\acute{N})$,
respectively. Moreover, we have%
\begin{eqnarray}
\breve{\nabla}_{W}U &=&\nabla _{W}U+h^{\ell }\left( W,U\right) +h^{s}\left(
W,U\right) ,\forall W,U\in \Gamma (T\acute{N}),  \label{9} \\
\breve{\nabla}_{W}N &=&-A_{N}W+\nabla _{W}^{\ell }N+D^{s}\left( W,N\right)
,N\in \Gamma (ltr(T\acute{N})),  \label{10} \\
\breve{\nabla}_{W}Z &=&-A_{Z}W+\nabla _{W}^{s}Z+D^{\ell }\left( W,Z\right)
,Z\in \Gamma (S(T\acute{N}^{\bot })).  \label{11}
\end{eqnarray}%
Denote the projection of $T\acute{N}$ on $S(T\acute{N})$ by $\breve{P}.$Then
by using (\ref{7}), (\ref{9})-(\ref{11}) and a metric connection $\breve{%
\nabla}$, we obtain%
\begin{eqnarray}
\bar{g}(h^{s}\left( W,U\right) ,Z)+\bar{g}\left( U,D^{\ell }\left(
W,Z\right) \right)  &=&g\left( A_{Z}W,U\right) ,  \label{12} \\
\bar{g}\left( D^{s}\left( W,N\right) ,Z\right)  &=&\bar{g}\left(
N,A_{Z}W\right) .  \label{13}
\end{eqnarray}%
From the decomposition of the tangent bundle of a lightlike submanifold, we
have%
\begin{eqnarray}
\nabla _{W}\breve{P}U &=&\nabla _{W}^{\ast }\breve{P}U+h^{\ast }(W,\breve{P}%
U),  \label{14} \\
\nabla _{W}\xi  &=&-A_{\xi }^{\ast }W+\nabla _{W}^{\ast t}\xi ,  \label{15}
\end{eqnarray}%
for $W,U\in \Gamma (T\acute{N})$ and $\xi \in \Gamma (RadT\acute{N}).$ By
using above equations, we obtain%
\begin{eqnarray}
g(h^{\ell }(W,\breve{P}U),\xi ) &=&g(A_{\xi }^{\ast }W,\breve{P}U),
\label{16} \\
g(h^{s}(W,\breve{P}U),N) &=&g(A_{N}W,\breve{P}U),  \label{17} \\
g(h^{\ell }(W,\xi ),\xi ) &=&0,\quad A_{\xi }^{\ast }\xi =0.  \label{18}
\end{eqnarray}%
In general, the induced connection $\nabla $ on $\acute{N}$ is not a metric
connection. Since $\breve{\nabla}$ is a metric connection, by using (\ref{9}%
) we get%
\begin{equation}
\left( \nabla _{W}g\right) \left( U,V\right) =\bar{g}\left( h^{\ell }\left(
W,U\right) ,V\right) +\bar{g}\left( h^{\ell }\left( W,V\right) ,U\right) .
\label{19}
\end{equation}%
However, to pay attention that $\nabla ^{\ast }$ is a metric connection on $%
S(T\acute{N})$.

\bigskip

Let $(N_{1},g_{1})$ and $(N_{2},g_{2})$ be two $m_{1}$ and $m_{2}$
-dimensional semi-Riemannian manifolds with constant indexes $q_{1}>0,$ $%
q_{2}>0,$ respectively. Let $\pi :N_{1}\times N_{2}\rightarrow N_{1}$ and $%
\sigma :N_{1}\times N_{2}\rightarrow N_{2}$ the projections which are given
by $\pi (w,u)=w$ and $\sigma (w,u)=u$ for $\ $any $(w,u)\in N_{1}\times
N_{2},$ respectively.

We denote the product manifold by $\breve{N}=(N_{1}\times N_{2},\breve{g}),$
where%
\begin{equation*}
\breve{g}(W,U)=g_{1}(\pi _{\ast }W,\pi _{\ast }U)+g_{2}(\sigma _{\ast
}W,\sigma _{\ast }U)
\end{equation*}%
for any $\forall W,U\in \Gamma (T\breve{N}).$ Then we have%
\begin{equation*}
\pi _{\ast }^{2}=\pi _{\ast },\newline
\pi _{\ast }\sigma _{\ast }=\sigma _{\ast }\pi _{\ast }=0,
\end{equation*}%
\begin{equation*}
\sigma _{\ast }^{2}=\sigma _{\ast }\newline
,\pi _{\ast }+\sigma _{\ast }=I,
\end{equation*}%
where $I$ is identity transformation. Thus $(\breve{N},\breve{g})$ is an $%
\left( m_{1}+m_{2}\right) $- dimensional semi-Riemannian manifold with
constant index $\left( q_{1}+q_{2}\right) .$ The semi-Riemannian product
manifold $\breve{N}=N_{1}\times N_{2}$ is characterized by $N_{1}$ and $N_{2}
$ are totally geodesic submanifolds of $\breve{N}.$

Now, if we put $F=\pi _{\ast }-\sigma _{\ast }$, then we can easily see that 
\begin{eqnarray*}
F.F &=&\left( \pi _{\ast }-\sigma _{\ast }\right) \left( \pi _{\ast }-\sigma
_{\ast }\right)  \\
F^{2} &=&\pi _{\ast }^{2}-\pi _{\ast }\sigma _{\ast }-\sigma _{\ast }\pi
_{\ast }+\sigma _{\ast }^{2}=I
\end{eqnarray*}%
\begin{equation*}
F^{2}=I,\newline
\bar{g}\left( FW,U\right) =\bar{g}\left( W,FU\right) 
\end{equation*}%
for any $W,U\in \Gamma (T\breve{N}).$ If we denote the levi-civita
connection on $\breve{N}$ by $\breve{\nabla},$ then it can be seen that $(%
\breve{\nabla}_{W}F)Y=0,$ for any $W,U\in \Gamma (T\breve{N}),$ that is, $F$
is parallel with respect to $\breve{\nabla}$ \cite{S}$.$

Let $(\breve{N},\breve{g})$ be a Semi-Riemannian manifold. Then $\breve{N}$
is called golden semi-Riemannian manifold if there exists an $(1,1)$ tensor
field $\breve{P}$ on $\breve{N}$ such that{\small \ }%
\begin{equation}
\breve{P}^{2}=\breve{P}+I  \label{20}
\end{equation}%
where $I$ is the identity map on $\breve{N}.$ Also, 
\begin{equation}
\breve{g}(\breve{P}W,U)=\breve{g}(W,\breve{P}U).  \label{21}
\end{equation}%
\newline
The semi-Riemannian metric (\ref{21}) is called $\breve{P}-$compatible and $(%
\breve{N},\breve{g},\breve{P})$ is named a golden semi-Riemannian manifold.
Also we have 
\begin{equation}
\breve{\nabla}_{W}\breve{P}U=\breve{P}\breve{\nabla}_{W}U  \label{21a}
\end{equation}%
\cite{BC}.

If $\breve{P}$ be a Golden structure, then (\ref{21}) is equivalent to 
\begin{equation}
\breve{g}(\breve{P}W,\breve{P}U)=\breve{g}(\breve{P}W,U)+\breve{g}(W,U)
\label{23}
\end{equation}%
for any $W,U\in \Gamma (T\breve{N}).$

If $F$ is almost product structure on $\breve{N}$, then%
\begin{equation}
\breve{P}=\frac{1}{\sqrt{2}}(I+\sqrt{5}F)  \label{24}
\end{equation}%
is a Golden structure on $\breve{N}.$ Conversely, if $\breve{P}$ is a Golden
structure on $\breve{N}$, then 
\begin{equation*}
F=\frac{1}{\sqrt{5}}(2\breve{P}-I)
\end{equation*}%
is an almost product structure on $\breve{N}$\cite{N}.

\section{Screen Transversal Lightlike Submanifolds Of Golden semi-Riemannian
Manifolds}

\begin{lemma}
Let $\acute{N}$ be a lightlike submanifold of golden semi-Riemannian
manifold. Also, let be%
\begin{equation*}
\breve{P}RadT\acute{N}\subseteq S(T\acute{N}^{\perp }).
\end{equation*}%
Then, $\breve{P}ltrT\acute{N}$ is also subvector bundle of screen
transversal vector bundle.
\end{lemma}

\begin{proof}
Let us accept the reversal of hypothesis. Namely, $ltrT\acute{N}$ is
invariant with respect to $\breve{P},$i.e. $\breve{P}ltrT\acute{N}=ltrT%
\acute{N}.$ From definition of lightlike submanifold, we have 
\begin{equation*}
\breve{g}(N,\xi )=1,
\end{equation*}%
for $\xi \in \Gamma (RadT\acute{N})$ and $N\in \Gamma (ltrT\acute{N}).$Also
from (\ref{23}), we find that%
\begin{equation*}
\breve{g}(N,\xi )=\breve{g}(\breve{P}N,\breve{P}\xi )=1.
\end{equation*}%
However, $\breve{P}N\in \Gamma (ltrT\acute{N})$ then by hypothesis, we get $%
\breve{g}(\breve{P}N,\breve{P}\xi )=0.$ Hence, we obtain a contradiction
which implies that $\breve{P}N$ does not belong to $ltrT\acute{N}$. Now, we
accept $\breve{P}N\in \Gamma (S(T\acute{N})).$ From here, we obtain that 
\begin{equation*}
1=\breve{g}(N,\xi )=\breve{g}(\breve{P}N,\breve{P}\xi )=0.
\end{equation*}%
So it is a contradiction. We get the same contradiction, when we assume $%
\breve{P}N\in \Gamma (RadT\acute{N}).$Thus, $\breve{P}N$ does not belong to $%
S(T\acute{N})$ and also $RadT\acute{N}.$Then, from the decomposition of a
lightlike submanifold, we conclude that $\breve{P}N\in \Gamma (S(T\acute{N}%
^{\perp })).$
\end{proof}

\begin{definition}
Let $\acute{N}$ be a lightlike submanifold of golden semi-Riemannian
manifold $\breve{N}.$ If 
\begin{equation*}
\breve{P}RadT\acute{N}\subset S(T\acute{N}^{\perp }),
\end{equation*}%
then, $\acute{N}$ is a screen transversal lightlike submanifold of $\breve{N}
$ golden semi-Riemannian manifold.
\end{definition}

\begin{definition}
Let $\acute{N}$ be a screen transversal lightlike submanifold of golden
semi-Riemannian manifold $\breve{N}.$ Then
\end{definition}

\begin{enumerate}
\item If $\breve{P}(S(T\acute{N}))=S(T\acute{N}),$ then, we say that $\acute{%
N}$ is a radical screen transversal lightlike submanifold of $\breve{N}.$

\item If $\breve{P}(S(T\acute{N}))\subset S(T\acute{N}^{\perp }),$ then, we
say that $\acute{N}$ is a screen transversal anti-invariant submanifold of $%
\breve{N}.$
\end{enumerate}

If $\acute{N}$ is a screen transversal anti-invariant submanifold of $\breve{%
N}$, then, we have 
\begin{equation*}
S(T\acute{N}^{\perp })=\breve{P}RadT\acute{N}\oplus \breve{P}ltrT\acute{N}%
\oplus \breve{P}(S(T\acute{N}))\perp D_{o}.
\end{equation*}%
In here $D_{o}$ is orthogonal complement non-degenerate distribution to $%
\breve{P}RadT\acute{N}\oplus \breve{P}ltrT\acute{N}\oplus \breve{P}(S(T%
\acute{N}))$.

\begin{proposition}
Let $\acute{N}$ be a screen transversal anti-invariant lightlike submanifold
of golden semi-Riemannian manifold $\breve{N}$. Then the distribution $D_{o}$
is invariant with respect to $\breve{P}.$
\end{proposition}

\begin{proof}
Using (\ref{21}), we obtain%
\begin{equation*}
\breve{g}(\breve{P}U,\xi )=g(U,\breve{P}\xi )=0,
\end{equation*}%
which show that $\breve{P}U$ does not belong to $ltrT\acute{N},$%
\begin{eqnarray*}
\breve{g}(\breve{P}U,N) &=&\breve{g}(U,\breve{P}N)=0, \\
\breve{g}(\breve{P}U,\breve{P}\xi ) &=&\breve{g}(U,\breve{P}\xi )+\breve{g}%
(U,\xi )=0, \\
\breve{g}(\breve{P}U,\breve{P}N) &=&0, \\
\breve{g}(\breve{P}U,W) &=&\breve{g}(U,\breve{P}W)=0, \\
\breve{g}(\breve{P}U,\breve{P}W) &=&0
\end{eqnarray*}%
for $U\in \Gamma (D_{o}),$ $\xi \in \Gamma (RadT\acute{N}),$ $N\in \Gamma
(ltrT\acute{N})$ and $W\in \Gamma (S(T\acute{N})).$ Therefore, the
distribution $D_{o}$ is invariant with respect to $\breve{P}.$
\end{proof}

Let $\acute{N}$ be screen transversal anti-invariant lightlike submanifold
of golden semi-Riemannian manifold $\breve{N}.$ Let $S$ and $R$ be
projection transformations of $S(T\acute{N})$ and $RadT\acute{N}$,
respectively. Then, for $W\in \Gamma (T\acute{N}),$ we have 
\begin{equation}
W=SW+RW,  \label{25}
\end{equation}%
on the other hand, if we apply $\breve{P}$ to (\ref{25}), we obtain%
\begin{equation*}
\breve{P}W=S_{1}W+S_{2}W,
\end{equation*}%
where $S_{1}W=\breve{P}SW\in \Gamma (S(T\acute{N}^{\perp })),$ $S_{2}W=%
\breve{P}RW\in \Gamma (S(T\acute{N}^{\perp }))$.

Let $\ P_{1},P_{2},P_{3},P_{4}$ be the projection morphisms on $\breve{P}RadT%
\acute{N},$ $\breve{P}ltrT\acute{N},$ $\breve{P}(S(T\acute{N})),$ $D_{o}$
respectively. Then, for $U\in \Gamma (S(T\acute{N}^{\perp })),$ we have 
\begin{equation}
U=P_{1}U+P_{2}U+P_{3}U+P_{4\text{ }}U.  \label{26}
\end{equation}%
If we apply $\breve{P}$ to (\ref{26}), then, we can find 
\begin{equation}
\breve{P}U=\breve{P}P_{1}U+\breve{P}P_{2}U+\breve{P}P_{3}U+\breve{P}P_{4%
\text{ }}U.  \label{27}
\end{equation}%
If we get $B_{1}=\breve{P}P_{1},$ $B_{2}=\breve{P}P_{2},$ $C_{1}=\breve{P}%
P_{3},$ $C_{2}=\breve{P}P_{4\text{ }}$, we can rewrite (\ref{27}) as%
\begin{equation*}
\breve{P}U=B_{1}U+B_{2}U+C_{1}U+C_{2}U.
\end{equation*}%
In here, there are components of $B_{1}U$ at$\ \Gamma (RadT\acute{N})$ with $%
\Gamma (S(T\acute{N}^{\perp }))$ and of $B_{2}U$ at $\Gamma (S(T\acute{N}))$
with $\Gamma (S(T\acute{N}^{\perp }))$ and of $C_{1}U$ at $\Gamma (ltrT%
\acute{N})$ with $\Gamma (S(T\acute{N}^{\perp }))$ and of $C_{2}U$ at $D_{o}$
with $\Gamma (S(T\acute{N}^{\perp })),$ Namely $\breve{P}U$ is belong to $T%
\breve{N}\mid _{\acute{N}}.$

It is known that the induced connection on a screen transversal
anti-invariant lightlike submanifold immersed in semi-Riemannian manifolds
is not a metric connection. The condition under which the induced connection
on the submanifold is a metric connection is given by the following theorem.

\begin{theorem}
Let $\acute{N}$ be a screen transversal anti-invariant lightlike submanifold
of golden semi-Riemannian manifold $\breve{N}$. Then, the induced connection 
$\nabla $ on $\acute{N}$ is a metric connection if and only if 
\begin{equation*}
K_{1}B_{1}\nabla _{W}^{s}\breve{P}\xi =0,
\end{equation*}%
for $W\in \Gamma (T\acute{N})$ and $\xi \in \Gamma (RadT\acute{N}).$
\end{theorem}

\begin{proof}
From (\ref{21a}), we have%
\begin{equation*}
\breve{\nabla}_{W}\breve{P}U=\breve{P}\breve{\nabla}_{W}U,
\end{equation*}%
in here, we get $U=\xi ,$ then we have%
\begin{eqnarray*}
\breve{\nabla}_{W}\breve{P}\xi  &=&\breve{P}\breve{\nabla}_{W}\xi  \\
-A_{\breve{P}\xi }W+\nabla _{W}^{s}\breve{P}\xi +D^{l}(W,\breve{P}\xi ) &=&%
\breve{P}\left( \nabla _{W}\xi +h^{l}(W,\xi \right) +h^{s}(W,\xi )).
\end{eqnarray*}%
Applying $\breve{P}$ to above equation, we find that%
\begin{equation*}
\left( 
\begin{array}{c}
-\breve{P}A_{\breve{P}\xi }W+B_{1}\nabla _{W}^{s}\breve{P}\xi +B_{2}\nabla
_{W}^{s}\breve{P}\xi  \\ 
+C_{1}\nabla _{W}^{s}\breve{P}\xi +C_{2}\nabla _{W}^{s}\breve{P}\xi +\breve{P%
}D^{l}(W,\breve{P}\xi )%
\end{array}%
\right) {\small =}\breve{P}^{2}\left( \nabla _{W}\xi +h^{l}(W,\xi \right)
+h^{s}(W,\xi )),
\end{equation*}%
in here, when we get, $K_{1}B_{1}\nabla _{W}^{s}\breve{P}\xi $ is remaining
part of $RadT\acute{N}$ of $B_{1}\nabla _{W}^{s}\breve{P}\xi $ and $%
K_{2}B_{2}\nabla _{W}^{s}\breve{P}\xi $ is remaining part of $S(T\acute{N})$
of $B_{2}\nabla _{W}^{s}\breve{P}\xi ,$ we can find 
\begin{equation*}
K_{1}B_{1}\nabla _{W}^{s}\breve{P}\xi =0.
\end{equation*}
\end{proof}

\begin{theorem}
Let $\acute{N}$ be a screen transversal anti-invariant lightlike submanifold
of golden semi-Riemannian manifold $\breve{N}$. Radical distribution is
integrable if and only if 
\begin{equation*}
\nabla _{W}^{s}\breve{P}U=\nabla _{U}^{s}\breve{P}W
\end{equation*}%
for $U,W\in \Gamma (RadT\acute{N}).$
\end{theorem}

\begin{proof}
From the definition of screen transversal anti-invariant lightlike
submanifold, the radical distribution is integrable if and only if 
\begin{equation*}
g\left( \left[ W,U\right] ,Z\right) =0
\end{equation*}%
for $U,W\in \Gamma (RadT\acute{N})$ and $Z\in \Gamma (S(T\acute{N})).$ From
here, we have 
\begin{equation*}
0=g(\breve{\nabla}_{W}\breve{P}U,\breve{P}Z)-g(\breve{\nabla}_{W}U,\breve{P}%
Z)-g(\breve{\nabla}_{U}\breve{P}W,\breve{P}Z)+g(\breve{\nabla}_{U}W,\breve{P}%
Z),
\end{equation*}%
because $\breve{P}U,\breve{P}W\in \Gamma (S(T\acute{N}^{\perp })),$ from (%
\ref{10}), we find%
\begin{equation*}
0=g(\nabla _{W}^{s}\breve{P}U-\nabla _{U}^{s}\breve{P}W,\breve{P}Z.
\end{equation*}%
Thus, proof is completed.
\end{proof}

\begin{theorem}
Let $\acute{N}$ be a screen transversal anti-invariant lightlike submanifold
of golden semi-Riemannian manifold $\breve{N}$. In this case, the screen
distribution\ is integrable if and only if 
\begin{equation*}
\nabla _{W}^{s}\breve{P}U-\nabla _{U}^{s}\breve{P}W=h^{s}(W,U)-h^{s}(U,W),
\end{equation*}%
for $W,U\in \Gamma (S(T\acute{N})).$
\end{theorem}

From the definition of a screen transversal anti-invariant lightlike
submanifold, the screen distribution is integrable if and only if 
\begin{equation*}
g\left( \left[ W,U\right] ,N\right) =0,
\end{equation*}%
$W,U\in \Gamma (S(T\acute{N}))$ and $N\in \Gamma (ltrT\acute{N}).$In here,
if we use (\ref{21a}) and (\ref{21}), we find%
\begin{eqnarray*}
0 &=&g(\breve{\nabla}_{W}\breve{P}U,\breve{P}N)-g(\breve{\nabla}_{W}U,\breve{%
P}N)-g(\breve{\nabla}_{U}\breve{P}W,\breve{P}N)+g(\breve{\nabla}_{U}W,\breve{%
P}N), \\
&=&g(\nabla _{W}^{s}\breve{P}U,\breve{P}N)-g(\nabla _{U}^{s}\breve{P}W,%
\breve{P}N)-g(h^{s}(W,U),\breve{P}N)+g(h^{s}(U,W),\breve{P}N){\small .}
\end{eqnarray*}%
From the last equation, we have%
\begin{equation*}
\nabla _{W}^{s}\breve{P}U-\nabla _{U}^{s}\breve{P}W=h^{s}(W,U)-h^{s}(U,W).
\end{equation*}%
Thus, the proof is completed.

\section{Radical Screen Transversal Lightlike Submanifolds Of Golden
Semi-Riemannian Manifolds}

\begin{theorem}
Let $\acute{N}$ be a \ radical screen transversal lightlike submanifold of
golden semi-Riemannian manifold $\breve{N}$. In this case, the screen
distribution is integrable if and only if 
\begin{equation*}
h^{s}(W,\breve{P}U)=h^{s}(U,\breve{P}W),
\end{equation*}%
$W,U\in \Gamma (S(T\acute{N})).$
\end{theorem}

\begin{proof}
By the definition of a radical screen transversal lightlike submanifold. The
screen distribution is integrable if and only if%
\begin{equation*}
g\left( \left[ W,U\right] ,N\right) =0,
\end{equation*}%
$W,U\in \Gamma (S(T\acute{N}))$ and $N\in \Gamma (ltrT\acute{N}).$ In here,
using (\ref{21}) and (\ref{21a}), we find%
\begin{eqnarray*}
0 &=&g(\breve{\nabla}_{W}\breve{P}U,\breve{P}N){\small -}g(\breve{\nabla}%
_{W}U,\breve{P}N){\small -}g(\breve{\nabla}_{U}\breve{P}W,\breve{P}N)+g(%
\breve{\nabla}_{U}W,\breve{P}N), \\
&=&g(h^{s}(W,\breve{P}U)-h^{s}(U,\breve{P}W)-h^{s}(W,U)+h^{s}(U,W),\breve{P}%
N),
\end{eqnarray*}%
from here, we have%
\begin{equation*}
h^{s}(W,\breve{P}U)-h^{s}(U,\breve{P}W)=h^{s}\left( W,U\right) -h^{s}\left(
U,W\right) .
\end{equation*}%
Therefore, the proof is completed.
\end{proof}

\begin{theorem}
Let $\acute{N}$ be a\ radical screen transversal lightlike submanifold of
golden semi-Riemannian manifold $\breve{N}$. The radical distribution is
integrable if and only if 
\begin{equation*}
A_{\breve{P}U}W-A_{\breve{P}W}U=A_{W}^{\ast }U-A_{U}^{\ast }W
\end{equation*}%
for $W,U\in \Gamma (RadT\acute{N}).$
\end{theorem}

\begin{proof}
From the radical screen transversal lightlike submanifold, the radical
distribution is integrable if and only if 
\begin{equation*}
g\left( \left[ W,U\right] ,Z\right) =0,
\end{equation*}%
for $W,U\in \Gamma (RadT\acute{N})$ and $Z\in \Gamma (S(T\acute{N})).$ Using
(\ref{21}) and (\ref{21a}), we have%
\begin{equation*}
0=g(\breve{\nabla}_{W}\breve{P}U,\breve{P}Z)-g(\breve{\nabla}_{U}\breve{P}W,%
\breve{P}Z)-g(\breve{\nabla}_{W}U,\breve{P}Z)+g(\breve{\nabla}_{U}W,\breve{P}%
Z).
\end{equation*}%
Since $\breve{P}U,\breve{P}W\in \Gamma (S(T\acute{N}^{\perp }))$ and $Z\in
\Gamma (S(T\acute{N}))$, if we use equation (\ref{9}) and (\ref{11}), we
obtain%
\begin{equation*}
0=g\left( A_{\breve{P}U}W-A_{\breve{P}W}U-A_{W}^{\ast }U+A_{U}^{\ast }W,%
\breve{P}Z\right) .
\end{equation*}%
Hence the proof is completed.
\end{proof}

\begin{proposition}
Let $\acute{N}$ be a radical screen transversal lightlike submanifold of
golden semi-Riemannian manifold $\breve{N}$. Then, the distribution $D_{o}$
is invariant with respect to $\breve{P}.$
\end{proposition}

\begin{proof}
In a radical screen transversal lightlike submanifold, we know that 
\begin{equation*}
S(T\acute{N}^{\perp })=\breve{P}RadT\acute{N}\oplus \breve{P}ltrT\acute{N}%
\perp D_{o},
\end{equation*}%
and we have $\breve{P}S(T\acute{N})=S(T\acute{N})$. From here, for $W\in
D_{o}$ and $U\in \Gamma (S(T\acute{N})),$ we find that%
\begin{eqnarray*}
g(\breve{P}W,\xi ) &=&g(W,\breve{P}\xi )=0, \\
g(\breve{P}W,\breve{P}\xi ) &=&0, \\
g(\breve{P}W,N) &=&g(W,\breve{P}N)=0, \\
g(\breve{P}W,\breve{P}N) &=&0, \\
g(\breve{P}W,U) &=&g(W,\breve{P}U)=0, \\
g(\breve{P}W,\breve{P}U) &=&0.
\end{eqnarray*}%
Therefore, we obtain that the distribution $D_{o}$ is invariant with respect
to $\breve{P}.$
\end{proof}

\begin{theorem}
Let $\acute{N}$ be a\ radical screen transversal lightlike submanifold of
golden semi-Riemannian manifold $\breve{N}$. The screen distribution define
totally geodesic foliation if and only if there is no component of $h^{s}(W,%
\breve{P}U)-h^{s}(W,U)$ in $\breve{P}ltrT\acute{N}$ for $W,U\in \Gamma (S(T%
\acute{N})).$
\end{theorem}

\begin{proof}
By definition of radical screen transversal lightlike submanifold, $S(T%
\acute{N})$ define totally geodesic foliation if and only if 
\begin{equation*}
\breve{g}\left( \nabla _{W}U,N\right) =0,
\end{equation*}%
where $W,U\in \Gamma (S(T\acute{N}))$ and $N\in \Gamma (ltrT\acute{N}).$
From here, if we use (\ref{21}) and (\ref{21a}), we can find%
\begin{equation*}
0=\breve{g}(\breve{\nabla}_{W}\breve{P}U,\breve{P}N)-\breve{g}(\breve{\nabla}%
_{W}U,\breve{P}N),
\end{equation*}%
from here, we have 
\begin{equation*}
0=g(h^{s}(W,\breve{P}U)-h^{s}(W,U),\breve{P}N).
\end{equation*}%
Hence, we obtain that, there is no component of $h^{s}(W,\breve{P}%
U)-h^{s}(W,U)$ in$\ \breve{P}ltrT\acute{N}$ .
\end{proof}

\begin{theorem}
Let $\acute{N}$ be a\ radical screen transversal lightlike submanifold of
golden semi-Riemannian manifold $\breve{N}$. The radical distribution define
totally geodesic foliation if and only if there is no component of $A_{%
\breve{P}U}W$\ in $S(T\acute{N})$ and $A_{U}^{\ast }W=0$, for $W,U\in \Gamma
(RadT\acute{N}).$
\end{theorem}

\begin{proof}
The radical distribution define totally geodesic foliation if and only if 
\begin{equation*}
\breve{g}\left( \nabla _{W}U,Z\right) =0,
\end{equation*}%
for $W,U\in \Gamma (RadT\acute{N})$ and $Z\in S(T\acute{N}),$ from here, we
have%
\begin{eqnarray*}
\breve{g}(\breve{\nabla}_{W}U,Z) &=&0\Leftrightarrow g(\breve{\nabla}_{W}%
\breve{P}U,\breve{P}Z)-g(\breve{\nabla}_{W}U,\breve{P}Z)=0 \\
&=&g(-A_{\breve{P}U}W+A_{U}^{\ast }W,\breve{P}Z).
\end{eqnarray*}%
Therefore, the proof is completed.
\end{proof}

\begin{theorem}
Let $\acute{N}$ be a\ radical screen transversal lightlike submanifold of
golden semi-Riemannian manifold $\breve{N}$. The induced connection on $%
\acute{N}$ is a metric connection if and only if there is no component of $%
h^{s}(U,W)$ in $\breve{P}RadT\acute{N}$ or of $A_{\breve{P}\xi }W$ in $S(T%
\acute{N})$ for $W,U\in S(T\acute{N}).$
\end{theorem}

\begin{proof}
For $W,U\in S(T\acute{N})$ and $\xi \in \Gamma (RadT\acute{N}),$we have%
\begin{equation*}
\breve{\nabla}_{W}\breve{P}\xi =\breve{P}\breve{\nabla}_{W}\xi .
\end{equation*}%
Taking inner product of this equation with $U\in \Gamma (S(T\acute{N})),$ we
obtain 
\begin{equation*}
\breve{g}(\breve{\nabla}_{W}\breve{P}\xi ,U)=\breve{g}(\breve{\nabla}_{W}\xi
,\breve{P}U),
\end{equation*}%
from here, if we use equation (\ref{11}), we can find that%
\begin{equation*}
-g(A_{\breve{P}\xi }W,U)=g(\nabla _{W}\xi ,\breve{P}U).
\end{equation*}%
Here, we come to the conclusion that, there is no component of $A_{\breve{P}%
\xi }W$\ in $S(T\acute{N})$ or using equation (\ref{12}) in last equation,
we have 
\begin{equation*}
-g\left( h^{s}(U,W),U\right) =g(\nabla _{W}\xi ,\breve{P}U).
\end{equation*}%
Namely, there is no component of $h^{s}(U,W)$ in $\breve{P}RadT\acute{N}$ $.$
\end{proof}

If $F$ is almost product structure on $\breve{N}$, then%
\begin{equation*}
\breve{P}=\frac{1}{\sqrt{2}}(I+\sqrt{5}F)
\end{equation*}%
is a Golden structure on $\breve{N}.$ Conversely, if $\breve{P}$ is a Golden
structure on $\breve{N}$, then 
\begin{equation*}
F=\frac{1}{\sqrt{5}}(2\breve{P}-I)
\end{equation*}%
is an almost product structure on $\breve{N}.$ We can give that follow
examples.

\begin{example}
Let $\breve{N}=%
\mathbb{R}
_{2}^{4}\times 
\mathbb{R}
_{2}^{4}$ be a semi-Riemannian product manifold with semi-Riemannian product
metric tensor $\bar{g}=\pi ^{\ast }g_{1}\otimes \sigma ^{\ast }g_{2},$ i=1,2
where $g_{i}$ denote standard metric tensor of $%
\mathbb{R}
_{2}^{4}$. Let we get%
\begin{eqnarray*}
f:\acute{N} &\rightarrow &\breve{N}, \\
(x_{1},x_{2},x_{3}) &\rightarrow
&(x_{1},\,x_{2}+x_{3},\,x_{1},\,0,\,x_{1},\,0,\,x_{2}-x_{3},\,x_{1}).
\end{eqnarray*}%
Then, we find 
\begin{eqnarray*}
Z_{1} &=&\partial x_{1}+\partial x_{3}-\partial x_{5}-\partial x_{8}, \\
Z_{2} &=&\partial x_{2}+\partial x_{7}, \\
Z_{3} &=&\partial x_{2}-\partial x_{7}.
\end{eqnarray*}%
The radical bundle $RadT\acute{N}$ is spanned by $Z_{1}$ and $\ S(T\acute{N}%
)=Span\left\{ Z_{2},Z_{3}\right\} $ for $FZ_{2}=Z_{3}.$ If we choose 
\begin{equation*}
N=-\frac{1}{2}(\partial x_{1}-\partial x_{5}),
\end{equation*}%
we find 
\begin{equation*}
g(Z_{1},N)=1.
\end{equation*}%
Also we obtain 
\begin{eqnarray*}
FZ_{1} &=&\partial x_{1}+\partial x_{3}+\partial x_{5}+\partial x_{8}=W_{1},
\\
FN &=&-\frac{1}{2}(\partial x_{1}+\partial x_{5})=W_{2}.
\end{eqnarray*}%
If we choose%
\begin{equation*}
W_{3}=\partial x_{4}+\partial x_{6},
\end{equation*}%
we obtain%
\begin{equation*}
FW_{3}=W_{4}=\partial x_{4}-\partial x_{6}.
\end{equation*}%
Thus we have $FRadT\acute{N}\subset S(T\acute{N}^{\perp })$ and $FS(T\acute{N%
})=S(T\acute{N}).$From here, using $\breve{P}=\frac{1}{\sqrt{2}}(I+\sqrt{5}%
F),$ we find that 
\begin{eqnarray*}
\breve{P}Z_{1} &=&\left( \frac{1}{\sqrt{2}}\left( 
\begin{array}{c}
(1+\sqrt{5})\partial x_{1}+(1+\sqrt{5})\partial x_{2} \\ 
+(\sqrt{5}-1)\partial x_{5}+(\sqrt{5}-1)\partial x_{8}%
\end{array}%
\right) \right) \in S{\small (}T\acute{N}^{\perp }), \\
\breve{P}N &=&{\small -}\frac{1}{2\sqrt{2}}((1+\sqrt{5})\partial x_{1}+(%
\sqrt{5}-1)\partial x_{5})\in S(T\acute{N}^{\perp }), \\
\breve{P}Z_{2} &=&\frac{1}{\sqrt{2}}((1+\sqrt{5})\partial x_{2}+(\sqrt{5}%
-1)\partial x_{7})\in S(T\acute{N}), \\
\breve{P}W_{3} &=&\frac{1}{\sqrt{2}}((1+\sqrt{5})\partial x_{4}+(\sqrt{5}%
-1)\partial x_{8})\in S(T\acute{N}).
\end{eqnarray*}%
Therefore, $\acute{N}$ is a radical screen transversal lightlike submanifold
of $\breve{N}.$
\end{example}

{\small \ }

\begin{example}
Let $\breve{N}=%
\mathbb{R}
_{2}^{4}\times 
\mathbb{R}
_{2}^{4}$ be a semi-Riemannian product manifold with semi-Riemannian product
metric tensor $\bar{g}=\pi ^{\ast }g_{1}\otimes \sigma ^{\ast }g_{2},$ \
i=1,2 where $g_{i}$ denote standard metric tensor of $%
\mathbb{R}
_{2}^{4}$. Let we get%
\begin{eqnarray*}
f &:&\acute{N}\rightarrow \breve{N}, \\
(x_{1},x_{2},x_{3}) &\rightarrow &(-x_{1},-x_{2},\,x_{1},\,\sqrt{2}%
x_{2},\,x_{1},\,0,\,0,\,x_{1},\,-x_{2}).
\end{eqnarray*}%
Then, we find 
\begin{eqnarray*}
Z_{1} &=&-\partial x_{1}+\partial x_{3}+\partial x_{5}+\partial x_{7} \\
Z_{2} &=&-\partial x_{2}-\sqrt{2}\partial x_{4}-\partial x_{8}.
\end{eqnarray*}%
The radical bundle $RadT\acute{N}$ is spanned by $Z_{1}$ and $\ S(T\acute{N}%
)=Span\left\{ Z_{2}\right\} $. If we choose 
\begin{equation*}
N=\frac{1}{2}(\partial x_{1}-\partial x_{5}),
\end{equation*}%
we find 
\begin{equation*}
g(Z_{1},N)=1.
\end{equation*}%
Thus, we obtain 
\begin{eqnarray*}
FZ_{1} &=&-\partial x_{1}+\partial x_{3}-\partial x_{5}-\partial x_{7}=W_{1},
\\
FZ_{2} &=&-\partial x_{2}-\sqrt{2}\partial x_{4}+\partial x_{8}=W_{2}, \\
FN &=&\frac{1}{2}(\partial x_{1}+\partial x_{5})=W_{3}.
\end{eqnarray*}%
If we choose%
\begin{equation*}
W_{4}=-\sqrt{2}\partial x_{2}-\sqrt{2}\partial x_{4}+\partial x_{6},
\end{equation*}%
we obtain%
\begin{equation*}
FW_{3}=W_{4}=\partial x_{4}-\partial x_{6}.
\end{equation*}%
Thus, we have $FRadT\acute{N}\subset S(T\acute{N}^{\perp })$ and $FS(T\acute{%
N})=S(T\acute{N}^{\perp }).$ From here, using $\breve{P}=\frac{1}{\sqrt{2}}%
(I+\sqrt{5}F),$ we find that%
\begin{eqnarray*}
\breve{P}Z_{1} &=&\left( \frac{1}{\sqrt{2}}\left( 
\begin{array}{c}
-(1+\sqrt{5})\partial x_{1}+(1+\sqrt{5})\partial x_{2} \\ 
+(1-\sqrt{5})\partial x_{5}+(1-\sqrt{5})\partial x_{7}%
\end{array}%
\right) \right) \in S(T\acute{N}^{\perp }), \\
\breve{P}Z_{2} &=&\left( \frac{1}{\sqrt{2}}\left( 
\begin{array}{c}
-(1+\sqrt{5})\partial x_{2}-\sqrt{2}(1+\sqrt{5})\partial x_{4} \\ 
+(\sqrt{5}-1)\partial x_{8}%
\end{array}%
\right) \right) \in S(T\acute{N}^{\perp }).
\end{eqnarray*}%
Therefore, $\acute{N}$ is a radical transversal lightlike submanifold.
\end{example}

\end{document}